\documentclass[11pt]{amsart}

\usepackage{amsmath,amsthm}
\usepackage{amsfonts,amsrefs}
\usepackage{verbatim}
\usepackage{amssymb}
\usepackage{txfonts}
\usepackage{amscd}
\usepackage{bbm}
\usepackage{hyperref}
\usepackage{mathrsfs}
\usepackage{fouriernc}
\usepackage{diagrams}

\newtheorem{theorem}{Theorem}

\newtheorem{lemma}[theorem]{Lemma}

\newcommand{\supp}{\operatorname{supp}}

\newcommand{\NN}{\mathbb{N}}

\title[1-cohomology is complemented]{Group 1-cohomology is complemented}

\author{Piotr W. Nowak}
\address{Institute of Mathematics, Polish Academy of Sciences, Warsaw, Poland -- and -- Institute of Mathematics, University of Warsaw, Poland}
\email{pnowak@impan.pl}
\thanks{The author was  partially supported by Narodowe Centrum Nauki grant DEC-2013/10/EST1/00352}

\begin{document}

\maketitle

\begin{abstract}
We show  
a structural property of cohomology with coefficients in an isometric representation on a uniformly convex Banach space:
if the  cohomology group $H^1(G,\pi)$ is reduced, then, up to an isomorphism, 
it is a closed complemented, subspace of the space of cocycles and its complement is the subspace of coboundaries. 
 \end{abstract}

Cohomology with coefficients in Banach modules has recently become an important object of study due to its use in the formulation 
of Banach space version of Kazhdan's property $(T)$, a property that the group 1-cohomology $H^1(G,\pi)$ with coefficients in
any isometric representation of $G$ on a fixed Banach space vanishes. See \cite{nowak-handbook} for a recent survey.
This notion generalizes the one of classical Kazhdan's property $(T)$ in the setting of 
unitary representations on Hilbert spaces, see \cite{bhv}. In that case several structural properties of cohomology
are obvious from the Hilbert space structure, however their Banach space counterparts are often non-trivial.
An example of such a property is complementability of closed 
subspaces appearing in the cohomological framework.
An important case is that of degree 0: the subspace of invariant vectors $E^\pi$ (i.e., the space of $0$-cocycles) 
of an isometric
representation $\pi$ of $G$ on $E$ 
is complemented in the representation space $E$,
whenever $E$ is reflexive \cite{brs,bfgm}. In that case the complement is a $\pi$-invariant subspace $E_\pi$. However, 
there are no other decompositions of this type known in higher degrees outside the setting of Hilbert spaces.

The purpose of this note is to prove the existence of a direct sum decomposition in degree 1. 
Consider a symmetric probability measure $\mu$ on $G$, whose support generates $G$ and which is given by a 
continous compactly supported density on $G$. By $A_\pi^\mu:E\to E$ denote the Markov operator associated to $\pi$ and $\mu$ 
via the Bochner integral
$$A_\pi^\mu v=\int_G \pi_gv \,d\mu.$$
By $Z^1(G,\pi)$ and $B^1(G,\pi)$ denote the space of $1$-cocycles and $1$-coboundaries, respectively.

\begin{theorem}\label{theorem : main}
Let $G$ be a locally compact, compactly generated group and $E$ be a Banach space. Assume that $\pi$ is a representation 
of $G$ on $E$, such that for some probability measure $\mu$ as above on $G$ the associated
Markov operator satisfies $\Vert A_\pi^\mu \Vert\le\lambda<1$. Then, up to isomorphism,
\begin{equation}\label{equation: direct sum decomposition}
Z^1(G,\pi)=B^1(G,\pi) \oplus H^1(G,\pi).
\end{equation}

In particular, \eqref{equation: direct sum decomposition} holds for every isometric representation $\pi$ of $G$
without almost invariant vectors on
a uniformly convex Banach space $E$.
\end{theorem}
We also show that certain equivariance in the above decomposition corresponds precisely to vanishing of cohomology.

\subsection*{Acknowledgements} I would like to thank the Referee for suggesting several improvements.

\section{Proofs}
Let $\pi$ be a strongly continuous representation of a compactly generated group $G$ on a Banach space $E$.
Let $Z^1(G,\pi)$ be the linear space of continuous cocycles for $\pi$ and let $S$, satisfying $\supp \mu\subseteq S$, be a  fixed compact symmetric generating set. We also assume that the support of $\mu$ generates $G$. 

A specific class of \emph{admissible measures} that is relevant for our purposes was introduced in \cite{drutu-nowak}.
A probability measure $\mu$ on $G$ is said to be admissible if it is defined by a density function $\rho$ (with respect to the Haar measure $dH$), satisfying 
$$\rho=\dfrac{\alpha+\beta}{\int_G \alpha+\beta\,dH },$$
where $\alpha,\beta$ are continuous, non-negative, compactly supported functions on $G$, such that 
\begin{enumerate}
\item $\alpha(e)>0$, where $e$ denotes the identity element of $G$;
\item $\int_G\alpha\ dH=1$;
 \item for every generator $s\in S$ we have $s\cdot \beta\ge \alpha$.
\end{enumerate}
Admissibility of a measure is a convenient tool when estimating spectral gaps of the associated Markov operator \cite{drutu-nowak}. 

A cocycle for $\pi$ is determined on the generators and we can define a norm on $Z^1(G,\pi)$ by the formula
$$\lVert z \rVert_S= \sup_{s\in S} \Vert z_s \Vert_E.$$
Let $d_\pi:E\to B^1(G,\pi)$ be the codifferential defined by $d_\pi v(g)=v-\pi_gv$. The subspace of coboundaries
$B^1(G,\pi)$ is the range of $d_\pi$ in $Z^1(G,\pi)$.

The representation $\pi$ does not have almost invariant vectors if there exists $\kappa=\kappa(\pi, S)>0$ such 
that $\sup_{s\in S}\Vert \pi_s v-v\Vert \ge \kappa \Vert v\Vert$ for every $v\in E$. In particular, 
if $\pi$ does not have almost invariant vectors then $B^1(G,\pi)$ is closed in the $\Vert \cdot \Vert_S$
norm in $Z^1(G,\pi)$. In that case the cohomology $H^1(G,\pi)=Z^1(G,\pi)/B^1(G,\pi)=Z^1(G,\pi)/\overline{B^1(G,\pi)}$
is said to be reduced.

\begin{proof}[Proof of Theorem \ref{theorem : main}] 
Let $\mu,\nu$ be  probability measures on $G$ as before, let $\mu*\nu$ denote the convolution  and $\mu^n$, $n\in \NN$, denote the convolution powers of $\mu$.
Let $z^{\mu}=\int_G z_g\,d\mu(g)$.
 Define an affine operator $L_{\pi,z}^{\mu}:E\to E$ by the formula
$$L_{\pi,z}^{\mu} v=A_\pi^\mu v + z^{\mu}.$$
\begin{lemma} Let $z,z'\in Z^1(G,\pi)$. Then 
\begin{enumerate}
\item $(z+z')^{\mu}=z^{\mu}+(z')^{\mu}$,
\item $L_{\pi,z+z'}^{\mu}=L_{\pi,z}^{\mu}+L_{\pi,z'}^{\mu}-A_\pi^\mu$,
\item $A_\pi^\mu z^{\nu}+z^{\mu}=z^{\mu*\nu}$,
\item $L_{\pi,z}^{\mu *\nu} =L_{\pi,z}^{\mu} L_{\pi,z}^{\nu}$. 
\end{enumerate}
\end{lemma}
\begin{proof}
The first two properties are clear, we only prove the remaining two. We have
\begin{align*}
A_\pi^\mu z^{\nu}+z^{\mu} &= \int_G \left( \int_G \pi_g z_h d\nu(h)\right) \,d\mu(g)+\int_G z_g \,d\mu(g)\\
&=\int_G\int_G z_{gh}\,d\nu(h)\,d\mu(g)\\
&=z^{\mu*\nu}.
\end{align*}
The last property is then a consequence of the equality
\begin{align*}
L_{\pi,z}^{\mu*\nu}v&= A_\pi^{\mu*\nu}v+z^{\mu*\nu} \\
&= A_\pi^\mu A_\pi^\nu v+A_\pi^\mu z^\nu + z^{\mu}\\
&=A_\pi^\mu(A_\pi^\nu v+ z^\nu) +z^{\mu}.
\end{align*}
\end{proof}

Under the assumptions of Theorem \ref{theorem : main} the map $L_{\pi,z}^{\mu}$ is a strict contraction. Indeed,
\begin{align*}
\left\lVert L_{\pi,z}^{\mu} v-L_{\pi,z}^{\mu} w\right\rVert &= \left\lVert A_\pi^\mu v + z^{\mu} -  A_\pi^\mu w- z^{\mu} \right\rVert\\
&\le \left\lVert A_\pi^\mu\right\rVert \left \lVert v-w\right\rVert.
\end{align*}
Since $\left\lVert A_\pi^\mu\right\rVert\le \lambda<1$, by the Banach Contraction Principle  $L_{\pi,z}^{\mu}$ has a  a unique fixed point 
$b(z)\in E$ such that $b(z)=L_{\pi,z}^{\mu} b(z).$
The assignment $z\mapsto b(z)$ is a map $b:Z^1(G,\pi)\to E$. Define a map $P:Z^1(G,\pi)\to B^1(G,\pi)$,
$$Pz=d_\pi b(z).$$
We claim that $P$ is a bounded projection onto $B^1(G,\pi)$.
\begin{lemma}
If $\left\lVert A_\pi^\mu \right\rVert<1$ then  $P$ is linear.
\end{lemma}
\begin{proof}
Recall that given $\mu$ and $z$, the fixed point $b(z)$ of $L_{\pi,z}^{\mu}$ is 
given by $$b(z)=\lim_{n\to \infty} \left(L_{\pi,z}^{\mu}\right)^n v=\lim_{n\to \infty}  z^{\left(\mu^n\right)}$$ for any $v\in E$.
Let $z,z'$ be cocycles for $\pi$.  Then
\begin{align*}
b(z+z') &= \lim_{n\to \infty} \left( L_{\pi,z+z'}^{\mu}\right)^nv\\
&= \lim_{n\to \infty} \left(L_{\pi, z}^{\mu}\right)^nv +\left(L_{\pi, z'}^{\mu}\right)^n v -\left(A_\pi^\mu\right)^nv \\
&= b(z)+b(z'),
\end{align*}
since $\left(A_\pi^\mu\right)^n v$ tends to 0 by the assumption that $\left\lVert A_\pi^\mu\right\rVert<1$.
\end{proof}

\begin{lemma}
If $\Vert A_\pi^\mu\Vert<1$ then $P$ is the identity on the range of $d_\pi$. 
\end{lemma}
\begin{proof}
Assume $z_g=v-\pi_gv$. Then
$$z^{\left(\mu^n\right)}=\int_G v-\pi_g v\, d\mu^n=v-\left(A_\pi^\mu\right)^n v,$$
and the last term tends to 0, so that $b(z)=v$ and $d_\pi b(z)=z$.
\end{proof}

\begin{lemma}
$P$ is continuous.
\end{lemma}
\begin{proof}
We estimate
\begin{align*}
\left\lVert z^{\mu} - b(z)\right\rVert_E &\le \sum_{n=1}^{\infty} d\left(L_{\pi,z}^{\left( \mu^n\right)} 0, L_{\pi,z}^{\left(\mu^{n+1}\right)} 0\right)\\
&\le  \lVert z^\mu\rVert_E \left( \sum_{n=1}^{\infty}\lambda^n \right).
\end{align*}
Therefore $\lVert b(z)\rVert_E \le  \lVert z^\mu \rVert_E \left( 1+ \sum_{n=1}^{\infty}\lambda^n \right)
\le  \lVert z\rVert_S \left( 1+ \sum_{n=1}^{\infty}\lambda^n \right)$ and finally
$$\Vert Pz\Vert_S =\Vert d_\pi b(z)\Vert_S \le 2\Vert b(z)\Vert_E \le 2  \lVert z\rVert_S \left( 1+ \sum_{n=1}^{\infty}\lambda^n \right).$$
\end{proof}
Let $V$ be the complement of $B^1(G,\pi)$ in $Z^1(G,\pi)$, then $V$ is isomorphic to $H^1(G,\pi)= Z^1(G,\pi) /B^1(G,\pi)$.
This proves the first statement of Theorem \ref{theorem : main}.

To see that the assumptions are always 
in the case of an isometric representation on a uniformly convex Banach space recall that it was shown 
in \cite{drutu-nowak} that for an admissible measure $\mu$ we have 
$\left\lVert A_\pi^\mu\right\Vert<1$ whenever the representation $\pi$ does not have almost invariant vectors.
\end{proof}

The above estimate of the norm of $P$ is strictly greater than 1. It would be interesting to know whether the coboundaries are 
in fact 1-complemented in $Z^1(G,\pi)$ in the case when $H^1(G,\pi)\neq 0$.

Denoting by $\gamma\cdot v=\pi_\gamma v+z_\gamma$ the affine actions associated to $\pi$ and $z\in Z^1(G,\pi)$
we have that
the following conditions are equivalent:
\begin{enumerate}
\item $H^1(G,\pi)=0$,
\item $b(\gamma \cdot z)=b(z)$ for every $\gamma\in G$, where $(\gamma\cdot z)_g= z_{\gamma g}$.
\end{enumerate}
Indeed, it is straightforward to verify that 
$$\pi_\gamma b(z)=b(\gamma\cdot z)- z_\gamma.$$
Thus 
$$z_\gamma=b(z)-\pi_\gamma b(z)=(Pz)_\gamma$$
if and only if $b(\gamma\cdot z)=b(z)$ for every $\gamma\in G$.

In the case when $G$ has property $(T)$ every cocycle associated to a unitary representation of $G$ is a coboundary.
Equivalently, $G$ has property $(T)$ if and only if every affine isometric action of $G$ on a Hilbert space has a fixed point.
See e.g. \cite[Chapter 2]{bhv}.
Therefore the above 
equivalent conditions are satisfied for every unitary representation $\pi$ and every cocycle.

\end{document}